\documentclass[8pt,twoside]{article}
\usepackage{mathrsfs}
\usepackage{amsmath}
\usepackage{amssymb}
\usepackage{fancyhdr}
\usepackage{latexsym}
\usepackage{bbding}
\usepackage{mathrsfs}
\usepackage{wasysym}

\usepackage{multicol,graphics}

\newtheorem{theorem}{Theorem}[section]
\newtheorem{lemma}[theorem]{Lemma}

\numberwithin{equation}{section}
\newenvironment{proof}[1][Proof]{\noindent\textbf{#1.} }{\hfill $\Box$}
\allowdisplaybreaks

 \makeatletter
\begin{document}
\title{{A Beale--Kato--Majda criterion for the 3-D Compressible Nematic Liquid Crystal Flows with Vacuum}
\thanks{Research supported by the National Natural Science Foundation of China (11171357).}
}
\author{{\small   Qiao Liu \thanks{\text{E-mail address}:  liuqao2005@163.com.
}}
 {\small\quad and\quad  Shangbin Cui \thanks{\text{E-mail
address}:
cuisb3@yahoo.com.cn.}}\\
{\small  Department of Mathematics, Sun Yat-sen
University, Guangzhou, Guangdong 510275,}\\
{\small People's Republic
 of China}\\
}
\date{}
\maketitle

\begin{abstract}
In this paper, we prove a Beale--Kato--Majda blow-up criterion in
terms of the gradient of the velocity only for the strong solution
to the 3-D compressible nematic liquid crystal flows with
nonnegative initial densities. More precisely, the strong solution
exists globally if the $L^{1}(0,T;L^{\infty})$-norm of the gradient
of the velocity $u$ is bounded. Our criterion improves the recent
result of X. Liu and L. Liu (\cite{LL}, A blow-up criterion for the
compressible liquid crystals system, arXiv:1011.4399v2 [math-ph] 23
Nov. 2010).
\medskip


\end{abstract}

\section{Introduction}\label{Int}

\noindent

The governing system of equations for the compressible nematic
liquid (NLC) crystal flows is the following system of scalar or
vector fields $\rho(t,x)$, $u(t,x)$ and $d(t,x)$ for $(t,x)\in
(0,+\infty)\times\Omega$, for a bounded smooth domain $\Omega\subset
\mathbb{R}^{3}$:
\begin{equation}
\label{NLC} \left\{
\begin{array}{l}
\partial_{t}\rho+\operatorname{div} (\rho u)=0,\\
{\partial_{t}}(\rho u) +\operatorname{div}(\rho u\otimes u)+\nabla{P}=\mu\Delta u-\lambda\nabla\cdot(\nabla d \odot\nabla d
-\frac{1}{2}(|\nabla d|^{2}+F(d)I)),\\
\partial_{t}d+(u\cdot\nabla)d=\nu(\Delta d-f(d))
\end{array}
\right.
\end{equation}
together with the initial value conditions:
\begin{align}\label{eq1.2}
\rho(0,x)=\rho_{0}(x)\geq 0, \quad u(0,x)=u_{0}(x),\quad
d(0,x)=d_{0}(x), \quad\forall x\in\Omega,
\end{align}
and the boundary value conditions:
\begin{align}\label{eq1.3}
u(t,x)=0,\quad d(t,x)=d_{0}(x),\quad |d_{0}(x)|=1, \quad\forall
(t,x)\in [0,+\infty)\times \partial\Omega.
\end{align}
Here we denote by $\rho$, $u=(u_{1},u_{2},u_{3})$,
$d=(d_{1},d_{2},d_{3})$ the unknown density, velocity and
orientation parameter of liquid crystal, respectively, and
$P=P(\rho)$ is the pressure function. Besides, $\mu,\lambda$ and
$\nu$ are positive viscosity coefficients. The non-standard term
$\nabla d\odot \nabla d$ denotes the $3\times 3$ matrix, whose
$(i,j)$-th element is given by
$\sum_{k=1}^{3}\partial_{i}d_{k}\partial_{j}d_{k}$. $I$ is the unit
matrix. $f(d)$ is a polynomial function of $d$ which satisfies
$f(d)=\frac{\partial}{\partial_{d}}F(d)$, where $F(d)$ is the bulk
part of the elastic energy; usually we choose $F(d)$ to be the
Ginzburg--Landan penalization, i.e.,
$F(d)=\frac{1}{4\sigma^{2}}(|d|^{2}-1)^{2}$ and
$f(d)=\frac{1}{\sigma^{2}}(|d|^{2}-1)d$, where $\sigma$ is a
positive constant. In what follows, we will assume $\sigma=1$ since
its specific value does not play a special role in our discussion.
Besides, we assume that the pressure function $P$ satisfies
\begin{align}\label{eq1.4}
P=P(\cdot)\in C^{1}[0,\infty),\quad P(0)=0.
\end{align}

The above system \eqref{NLC} is a simplified version of
Ericksen--Lesile system modeling the flow of compressible nematic
liquid crystals, and the hydrodynamic theory of liquid crystals was
established by Ericksen \cite{ER1,ER2} and Leslie \cite{LE} in the
1960's.  When $d\equiv 0$, the system becomes to the compressible
Navier--Stokes (CNS) equations. Matsumura and Nishida \cite{MN}
obtained global existence of smooth solutions for the initial data
is a small perturbation of a non--vacuum equilibrium. For the
existence of solutions for arbitrary initial value, Lions \cite{PL}
and Feireisl \cite{EF} established the global existence of weak
solution to the CNS equations. Cho et al.\cite{CK1,CCK,CK2} proved
that the existence and uniqueness of local strong solutions of the
CNS equations in the case where initial density need not to be
positive  and may vanish in an open set. Xin in \cite{X} showed that
there is no global smooth solution to the Cauchy problem of the CNS
equations with a nontrivial compactly supported initial density.
Hence, there are many works \cite{CCK,FJ,FJO,HX,HLX1,HLX2,SWZ1,SWZ2}
try to establish blow--up criterion for the strong solution to the
CNS equations. In particular, it is proved in \cite{HLX2} by Huang,
Li and Xin that the serrin's blow--up criterion (see \cite{JS}) for
the incompressible Navier--Stokes equations still holds for the CNS
equations, i.e., if $T^{*}$ is the maximal time of existence strong
solution, then
\begin{align}\label{eq1.5}
\lim_{T\rightarrow T^{*}}(\|\operatorname{div}
u\|_{L^{1}(0,T;L^{\infty})}+\|\rho^{\frac{1}{2}}
u\|_{L^{s}(0,T;L^{r})})=\infty
\end{align}
or
\begin{align}\label{eq1.6}
\lim_{T\rightarrow
T^{*}}(\|\rho\|_{L^{1}(0,T;L^{\infty})}+\|\rho^{\frac{1}{2}}
u\|_{L^{s}(0,T;L^{r})})=\infty,
\end{align}
where $r$ and $s$ satisfy $\frac{2}{s}+\frac{3}{r}\leq 1$, $3< r\leq
\infty$. In \cite{HX,HLX1}, Huang et al. established that the
Beale--Kato--Majda criterion (see \cite{BKM}) for the ideal
incompressible flows still hold for the CNS equations, that is
\begin{align*}
\lim_{T\rightarrow T^{*}}\int_{0}^{T}\|\nabla
u\|_{L^{\infty}}\text{d} t=\infty.
\end{align*}
Sun, Wang and Zhang in \cite{SWZ1} (see also \cite{HLX2}) obtained
another Beale--Kato--Majda criterion in terms of the density, i.e.,
\begin{align*}
\lim\sup_{T\rightarrow T^{*}}\|\rho
\|_{L^{\infty}(0,T;L^{\infty})}=\infty.
\end{align*}

When $\rho$ is a positive constant, the system \eqref{NLC} becomes
to the incompressible nematic liquid crystal (INLC) equations, the
global-in-time weak solutions and local-in-time strong solution have
been studied by Lin and Liu \cite{LL1,LL2}. In \cite{HW}, Hu and
Wang established global existence of strong solutions and
weak--strong uniqueness for initial data belonging to the Besov
spaces of positive order under some smallness assumptions. Liu and
Cui in \cite{LC} obtained that the blow--up criterion \eqref{eq1.5}
or \eqref{eq1.6} still holds for the solution of the INLC equations.
We also refer \cite{HKL,L,LW1,LW,SL} and the reference cited therein
for other related work on the INLC equations.

Inspired by the above mentioned works on blow--up criterion of
strong solution of CNS and INLC equations, particularly the results
of Huang etal. \cite{HX,HLX1} and Sun et al. \cite{SWZ1,SWZ2}, we
want to investigate a similar problem for the compressible nematic
liquid crystal flow \eqref{NLC}--\eqref{eq1.3}. Before stating the
main result, we denote the following simplified notations of Sobolev
spaces
\begin{align*}
L^{q}:=L^{q}(\Omega),\quad W^{k,p}:=W^{k,p}(\Omega),\quad
H^{k}:=H^{k}(\Omega),\quad H_{0}^{1}:=H_{0}^{1}(\Omega).
\end{align*}

When the initial vacuum is allowed, the well-posedness and blow--up
criterion for strong solutions to the compressible nematic liquid
crystal flows \eqref{NLC}--\eqref{eq1.3} have been investigated by
Liu et al. in \cite{LL,LLH}. Here, we write down the main results of
Liu et al. \cite{LL,LLH}.

\begin{theorem}\label{thm1.1}
Suppose that the initial value $(\rho_{0},u_{0},d_{0})$ satisfies
the following regularity conditions
\begin{align*}
0\leq \rho_{0}\in W^{1,6},\quad u_{0}\in H_{0}^{1}\cap H^{2}
\quad\text{ and }\quad d_{0}\in H^{3},
\end{align*}
and the compatibility condition
\begin{align}\label{eq1.7}
\mu \Delta u_{0}-\lambda \operatorname{div} (\nabla d_{0}\odot\nabla
d_{0}-\frac{1}{2}(|\nabla d_{0}|^{2}+F(d_{0})))-\nabla
P(\rho_{0})=\sqrt{\rho}g \text{ for some } g\in L^{2}.
\end{align}
Then there exist a small $T\in (0,\infty)$ and a unique strong
solution $(\rho,u,d)$ to the system \eqref{NLC} with initial
boundary condition \eqref{eq1.2}--\eqref{eq1.3} such that
\begin{align*}
&0\leq \rho\in C([0,T);W^{1,6}),\quad\quad\quad\quad\quad\quad \rho_{t}\in C([0,T);L^{6}),\\
&u\in C([0,T);H^{1}_{0}\cap H^{2})\cap L^{2}(0,T;W^{2,6}),\quad
u_{t}\in L^{2}(0,T;H^{1}_{0}),\\
&d\in C([0,T);H^{3}),\quad\quad\quad\quad\quad\quad d_{t}\in
C([0,T);H^{1}_{0})\cap
L^{2}(0,T;H^{2}),\\
&d_{tt}\in L^{2}(0,T;L^{2}),\quad\quad\quad\quad\quad\quad
\sqrt{\rho}u_{t}\in C([0,T);L^{2}).
\end{align*}
Moreover, let $T^{*}$ be the maximal existence time of the solution.
If $T^{*}<\infty$, then there holds
\begin{align}\label{eq1.8}
\lim_{T\rightarrow T^{*}} \int_{0}^{T}(\|\nabla
u\|^{\beta}_{L^{\alpha}}+\|u\|_{W^{1,\infty}})\text{d}t=\infty,
\end{align}
where $\alpha,\beta$ satisfying $\frac{3}{\alpha}+\frac{2}{\beta}<2$
and $\beta\geq 4$.
\end{theorem}
\textbf{Remark 1.1} Another similar system of partial differential
equations modeling compressible nematic liquid crystal flows has
been studied by Huang, Wang and Wen in \cite{HWW1,HWW2}. They
obtained the existence of local in time strong solution and two
blow--up criteria under some suitable assumption condition $u$ and
$d$ or $\rho$ and $d$.
\medskip

The purpose of this paper is to obtain the Beale--Kato--Majda
blow--up criterion only in terms of the gradient of the velocity
still holds for the liquid crystal flows. Our main result is the
following

\begin{theorem}\label{thm1.2}
Assume that $(\rho, u, d)$ is the strong solution constructed in
Theorem \ref{thm1.1}, and $T^{*}$ be the maximal existence time of
the solution. If $T^{*}<\infty$, then we have
\begin{align}\label{eq1.9}
\limsup_{T\rightarrow T^{*}}\|\nabla
u\|_{L^{1}(0,T;L^{\infty})}=\infty.
\end{align}
\end{theorem}

The proof of this theorem will be given in the next section. As a
standard practice, we will show that if (\ref{eq1.9}) does not hold
then the strong solution $(\rho,u,d)$ can be extended beyond the
time $T^*$. To this end we will step-by-step establish a series of
higher-order norm estimates for the strong solution $(\rho,u,d)$.
The key fact used in this deduction is that the boundedness of the
$L^{1}(0,T; L^{\infty})$-norm of $\nabla u$ implies both boundedness
of the $L^{\infty}(0,T;L^{\infty})$-norm of the density $\rho$ and
boundedeness of the $L^{\infty}(0,T;W^{1,q})$-norm of $d$ with
$2\leq q\leq \infty$.

\section{Proof of Theorem \ref{thm1.2}}

Let $(\rho, u, d)$ be the unique strong solution to the system
\eqref{NLC} with initial--boundary condition
\eqref{eq1.2}--\eqref{eq1.3}. We assume that the opposite to
\eqref{eq1.9} holds, i.e.,
\begin{align}\label{eq2.1}
\lim_{T\rightarrow T^{*}} \|\nabla u\|_{L^{1}(0,T;L^{\infty})}\leq
M<\infty.
\end{align}
In what follows, we note that $C$ denotes a generic constant
depending only on $\mu,\lambda,\nu, M, T,\Omega$ and the initial
data. By using the mass conservation equation $\eqref{NLC}_{1}$ and
the assumption \eqref{eq2.1}, it is easy to obtain the
$L^{\infty}$-norm bounds of the density,

\begin{lemma}\label{lem2.1}
Assume that
\begin{align}\label{eq2.2}
\int_{0}^{T}\|\operatorname{div}u\|_{L^{\infty}}\text{d}t\leq
C,\quad \quad 0\leq T< T^{*},
\end{align}
then
\begin{align}\label{eq2.3}
\|\rho\|_{L^{\infty}(0,T;L^{\infty})}\leq C\quad\quad \forall 0 \leq
T<T^{*}.
\end{align}
\end{lemma}

\begin{proof}
The proof is essentially due to Huang and Xin \cite{HX}, for
reader's convenience, we sketch it here.

Multiplying the mass conservation equation $\eqref{NLC}_{1}$ by
$q\rho^{q-1}$ with $q>1$, it follows that
\begin{align*}
\partial_{t}(\rho^{q})
+\operatorname{div}(\rho^{q}u)+(q-1)\rho^{q}\operatorname{div}u=0.
\end{align*}
Integrating the above equality over $\Omega$ yields
\begin{align*}
\partial_{t}\|\rho\|_{L^{q}}^{q}\leq
(q-1)\|\operatorname{div}u\|_{L^{\infty}}\|\rho\|_{L^{q}}^{q},
\end{align*}
i.e.,
\begin{align}\label{eq2.4}
\partial_{t}\|\rho\|_{L^{q}}\leq
\frac{(q-1)}{q}\|\operatorname{div}u\|_{L^{\infty}}\|\rho\|_{L^{q}}.
\end{align}
The condition \eqref{eq2.2} and the estimate \eqref{eq2.3} imply
that
\begin{align*}
\partial_{t}\|\rho\|_{L^{q}}\leq
C\quad \text{ for } \forall q>1,
\end{align*}
where $C$ is a positive constant independent of $q$, letting
$q\rightarrow\infty$, we obtain \eqref{eq2.3}, and this completes
the proof of the lemma.
\end{proof}
\medskip

According to the assumption \eqref{eq1.4} on the pressure $P$ and
Lemma \ref{lem2.1}, it is easy to obtain
\begin{align}\label{eq2.5}
\sup_{0\leq t\leq
T}\{\|P(\rho)\|_{L^{\infty}},\|P'(\rho)\|_{L^{\infty}}\}\leq
C<\infty.
\end{align}

Now, let us derive the stand energy inequality.
\begin{lemma}\label{lem2.2}
There holds
\begin{align}\label{eq2.6}
\sup_{0\leq t\leq T}\int_{\Omega}(\rho|u|^{2}+|\nabla
d|^{2}+2F(d))\text{d}x+\int_{0}^{T}\int_{\Omega}|\nabla
u|^{2}\text{d}x\text{d}t+\int_{0}^{T}\int_{\Omega} |\Delta
d-f(d)|^{2}\text{d}x\text{d}t\leq C.
\end{align}
\end{lemma}

\begin{proof}
Multiplying the momentum equation $\eqref{NLC}_{2}$ by $u$,
integrating over $\Omega$ and making use of the mass conversation
equation $\eqref{NLC}_{1}$, it follows that
\begin{align}\label{eq2.7}
\frac{1}{2}\frac{d}{dt}\int_{\Omega}\rho|u|^{2}\text{d}x+\mu\int_{\Omega}|\nabla
u|^{2}\text{d}x=-\int_{\Omega} u\nabla
P\text{d}x-\lambda\int_{\Omega}(u\cdot\nabla)d\cdot(\Delta d
-f(d))\text{d}x,
\end{align}
where we have used the fact that $\operatorname{div}(\nabla d\odot
\nabla d)=(\nabla d)^{T}\Delta d-\nabla\cdot\frac{|\nabla
d|^{2}}{2}$. Multiplying the liquid crystal equation
$\eqref{NLC}_{3}$ by $\Delta d-f(d)$ and integrating over $\Omega$,
we obtain
\begin{align}\label{eq2.8}
\frac{d}{dt}\int_{\Omega}(\frac{1}{2}|\nabla
d|^{2}+F(d))\text{d}x+\nu \int_{\Omega}|\Delta
d-f(d)|^{2}\text{d}x=\int_{\Omega}(u\cdot \nabla)d\cdot(\Delta
d-f(d))\text{d}x.
\end{align}
Combining \eqref{eq2.7} and \eqref{eq2.8} together
\begin{align}\label{eq2.9}
&\frac{d}{dt}\int_{\Omega}[\frac{1}{2}(\rho |u|^{2}+\lambda |\nabla
d|^{2})+\lambda F(d)]\text{d}x+\mu\int_{\Omega }|\nabla
u|^{2}\text{d}x+\lambda\nu \int_{\Omega }|\Delta
d-f(d)|^{2}\text{d}x\nonumber\\
=& -\int_{\Omega } u\nabla
P\text{d}x=\int_{\Omega}P\operatorname{div}u\text{d}x\leq
\varepsilon\int_{\Omega }|\nabla u|^{2}\text{d}x+C \varepsilon^{-1},
\end{align}
where we have used the estimates \eqref{eq2.3}, \eqref{eq2.5} and
the Young inequality. Taking $\varepsilon$ small enough and applying
the Gronwall's inequality, we can establish the estimate
\eqref{eq2.6} immediately.
\end{proof}
\medskip

In the next lemma, we will derive some estimates of $d$.

\begin{lemma}\label{lem2.3}
Under the assumption \eqref{eq2.1}, it holds that for $0\leq T<
T^{*}$
\begin{align}
    \label{eq2.10}
&\sup_{0\leq t\leq T}(\|d\|_{L^{q}}+\|\nabla d\|_{L^{q}})\leq C\quad
\text{ for all } 2\leq q\leq \infty;\\
    \label{eq2.11}
&\sup_{0\leq t\leq T}\|\nabla d\|_{L^{2}}^{2}+\int_{0}^{T}\|
d_{t}\|_{L^{2}}^{2}\text{d}t\leq C;
\end{align}
\end{lemma}

\begin{proof}
We first multiplying the liquid crystal equation $\eqref{NLC}_{3}$
by $q|d|^{q-2}d$ with $q\geq 2$, and integrating over $\Omega$, then
there holds
\begin{align}
&\frac{d}{dt}\|d\|_{L^{q}}^{q}+\int_{\Omega}(q\nu |\nabla
d|^{2}|d|^{2}+q(q-2)\nu|d|^{q-2}|\nabla
|d||^{2})\text{d}x\nonumber\\
= &-\sum_{i=1}^{3}\int_{\Omega} u_{i}\partial_{i}
(|d|^{q})\text{d}x-q\nu\int_{\Omega }
|d|^{q+2}\text{d}x+q\nu\int_{\Omega}|d|^{q}\text{d}x\nonumber\\
= &-\sum_{i=1}^{3}\int_{\Omega} \partial_{i}u_{i}
|d|^{q}\text{d}x-q\nu\int_{\Omega }
|d|^{q+2}\text{d}x+q\nu\int_{\Omega}|d|^{q}\text{d}x\nonumber\\
\leq & C(\|\nabla u\|_{L^{\infty}}+1)\|d\|_{L^q}^{q}.\nonumber
\end{align}
By using the Gronwall's inequality, one obtains the inequality
\begin{align}\label{eq2.12}
\sup_{0\leq t\leq T}\|d\|_{L^{q}}\leq C\quad\text{ for all }q\geq 2.
\end{align}
By letting $q\rightarrow\infty$, we notice that the estimate
\eqref{eq2.12} still holds.

Multiplying the gradient of the liquid crystal equation
$\eqref{NLC}_{3}$ by $q|\nabla d|^{q-2}\nabla d$ with $q\geq 2$, and
integrating over $\Omega$, then there holds
\begin{align*}
&\frac{d}{dt}\|\nabla d\|_{L^{q}}^{q}+\int_{\Omega}(q\nu |\nabla
(\nabla d)|^{2}|\nabla d|^{q-2}+q(q-2)\nu |\nabla |\nabla
d||^{2}|\nabla d|^{q-2})\text{d}x\nonumber \\
= \!&\!-\!\sum_{i=1}^{3}\!\int_{\Omega}\!u_{i}\partial_{i}(|\nabla
d|^{q}\!)\text{d}x\!-\!\sum_{i=1}^{3}\!q\!\int_{\Omega}\!\nabla
u_{i}\partial_{i}d |\nabla d|^{q-2}\nabla d \text{d}x\!-\!\nu
q\!\int_{\Omega}\!\nabla(|d|^{2}d)|\nabla d|^{q-2}\nabla
d\text{d}x\!+\!\nu q\!
\int_{\Omega}\!|\nabla d|^{q}\text{d}x\nonumber\\
= \!&\!\!\sum_{i=1}^{3}\!\int_{\Omega}\partial_{i}u_{i}|\nabla
d|^{q}\text{d}x\!-\!\sum_{i=1}^{3}\!q\!\int_{\Omega }\!\nabla
u_{i}\partial_{i}d |\nabla d|^{q-2}\nabla d \text{d}x\!-\!\nu
q\!\int_{\Omega}\!\nabla(|d|^{2}d)|\nabla d|^{q-2}\nabla
d\text{d}x\!+\!\nu q\!
\int_{\Omega}\!|\nabla d|^{q}\text{d}x\nonumber\\
\leq& C\|\nabla u\|_{L^{\infty}}\|\nabla d\|_{L^{q}}^{q}+\nu
q\|\nabla d\|_{L^{q}}^{q}-\!\nu
q\!\int_{\Omega}\!\nabla(|d|^{2}d)|\nabla d|^{q-2}\nabla
d\text{d}x\nonumber\\
=&(C\|\nabla u\|_{L^{\infty}}+\nu q)\|\nabla d\|_{L^{q}}^{q}-\nu
q\int_{\Omega}|d|^{2}\nabla d|\nabla d|^{q-2}\nabla d\text{d}x-\nu
q\int_{\Omega}d\nabla (|d|^{2})|\nabla d|^{q-2}\nabla
d\text{d}x\nonumber\\
=&(C\|\nabla u\|_{L^{\infty}}+\nu q)\|\nabla d\|_{L^{q}}^{q}-3\nu
q\int_{\Omega}|d|^{2}|\nabla d|^{q}\text{d}x\nonumber\\
\leq &C(\|\nabla u\|_{L^{\infty}}+1)\|\nabla d\|_{L^{q}}^{q},
\end{align*}
where we have used the fact that $\nabla |d|^{2}=2|d|\nabla
|d|=2|d|\frac{d\cdot\nabla d}{|d|}=2d\nabla d$ in the last equality.
By using the Gronwall's inequality again, we obtain
\begin{align}\label{eq2.13}
\sup_{0\leq t\leq T}\|\nabla d\|_{L^{q}}\leq C\quad \text{ for all }
q\geq 2.
\end{align}
Letting $q\rightarrow\infty$, estimate \eqref{eq2.13} still holds,
and the inequalities \eqref{eq2.12} and \eqref{eq2.13} imply that
estimate \eqref{eq2.10} holds.

To prove the estimate \eqref{eq2.11}, we multiplying the liquid
crystal equation $\eqref{NLC}_{3}$ by $ d_{t}$ and integrating over
$\Omega$, then
\begin{align*}
&\|d_{t}\|_{L^{2}}^{2}+\frac{\nu}{2}\frac{d}{dt}\int_{\Omega}|\nabla
d|^{2}\text{d}x=-\int_{\Omega}(u\cdot \nabla)d
d_{t}\text{d}x-\nu\int_{\Omega}f(d)d_{t}\text{d}x\nonumber\\
\leq & C(\|u\|_{L^{2}}\|\nabla
d\|_{L^{\infty}}\|d_{t}\|_{L^{2}}+\|d\|_{L^{\infty}}^{2}
\|d\|_{L^{2}}\|d_{t}\|_{L^{2}}+\|d\|_{L^{2}}\|d_{t}\|_{L^{2}})\nonumber\\
\leq& \frac{1}{2}\|d_{t}\|_{L^{2}}^{2}+C,
\end{align*}
where we have used the estimates \eqref{eq2.6} and \eqref{eq2.10}.
Integrating the above inequality over $[0,T]$ gives the estimate
\eqref{eq2.11}.
\end{proof}
\medskip

For function $f\in \Omega\times (0,T)$, let
\begin{align*}
\dot{f}=f_{t}+u\cdot\nabla f
\end{align*}
denote the material derivative of the function $f$. Then we have
following lemma.

\begin{lemma}\label{lem2.4}
Under the assumption \eqref{eq2.1}, it holds that for $0\leq T<
T^{*}$
\begin{align}\label{eq2.14}
\sup_{0\leq t\leq T}(\|\nabla u\|_{L^{2}}^{2}+&\|\nabla
\rho\|_{L^{2}}^{2}+\|
d\|_{H^{2}}^{2})+\int_{0}^{T}\int_{\Omega}(\rho
|\dot{u}|^{2}+|\nabla d_{t}|^{2})\text{d}x\text{d}t\leq C;\\
       \label{eq2.15}
&\int_{0}^{T}\|\nabla d \|_{H^{2}}^{2}\text{d}t\leq C.
\end{align}

\end{lemma}

\begin{proof}
Noticing that the momentum equation $\eqref{NLC}_{2}$ can be rewrote
as
\begin{align}\label{eq2.16}
\rho \dot{u} +\nabla P=\mu\Delta u-\lambda (\nabla d)^{T}(\Delta
d-f(d)).
\end{align}
Multiplying the equation \eqref{eq2.16} above by $\dot{u}$ and
integrating over $\Omega$, one obtains the equality
\begin{align}\label{eq2.17}
&\frac{\mu}{2}\frac{d}{dt}\|\nabla u\|_{L^{2}}^{2}+\int_{\Omega
}\rho|\dot{u}|^{2}\text{d}x\nonumber\\
=&\mu\int_{\Omega} u\cdot\nabla u\Delta
u\text{d}x+\int_{\Omega}P\operatorname{div}u_{t}\text{d}x-\int_{\Omega}u\cdot\nabla
u\nabla
P\text{d}x\nonumber\\
&-\lambda\int_{\Omega}(u_{t}\cdot\nabla)d(\Delta
d-f(d))\text{d}x-\lambda\int_{\Omega}(u\cdot\nabla) u \cdot \nabla
d(\Delta d-f(d))\text{d}x
\end{align}
Combining the mass conservation equation $\eqref{NLC}_{1}$ and the
assumption \eqref{eq2.1}, it follows that the pressure $P$ satisfies
the following equation
\begin{align}\label{eq2.18}
P_{t}+P'(\rho)\nabla\rho\cdot u+P'(\rho)\rho\operatorname{div} u=0.
\end{align}
Hence, we have
\begin{align}\label{eq2.19}
\int_{\Omega}P\operatorname{div}u_{t}\text{d}x=&\frac{d}{dt}
\int_{\Omega}P\operatorname{div}u\text{d}x-\int_{\Omega}P_{t}\operatorname{div}u\text{d}x\nonumber\\
=&\frac{d}{dt} \int_{\Omega}P\operatorname{div}u\text{d}x+
\int_{\Omega}P'(\rho)(\nabla\rho \cdot
u+\rho\operatorname{div}u)\operatorname{div}u\text{d}x.
\end{align}
To estimate the term
$-\lambda\int_{\Omega}(u_{t}\cdot\nabla)d(\Delta d-f(d))\text{d}x$,
we have
\begin{align}\label{eq2.20}
-\!\lambda\int_{\Omega}\!(u_{t}\cdot\nabla)d&(\Delta
d-f(d))\text{d}x=\lambda\!\sum_{i,j=1}^{3}\!(\!\int_{\Omega}\!\partial_{j}u_{it}\partial_{i}d\partial_{j}d\text{d}x
+\!\int_{\Omega}\!u_{it}\partial_{i}\partial_{j}d\partial_{j}d\text{d}x)+\lambda
\int_{\Omega}u_{t}\cdot\nabla d f(d)\text{d}x\nonumber\\
=&\lambda\!\sum_{i,j=1}^{3}\!(\!\int_{\Omega}\!\partial_{j}u_{it}\partial_{i}d\partial_{j}d\text{d}x
-\frac{1}{2}\!\int_{\Omega}\!\partial_{i}u_{it}|\partial_{j}d|^{2}\text{d}x)-\lambda
\sum_{i=1}^{3}\int_{\Omega}\partial_{i}u_{it}(\frac{|d|^{4}}{4}-\frac{|d|^{2}}{2})\text{d}x\nonumber\\
=&\lambda\!\sum_{i,j=1}^{3}\{\frac{d}{dt}\!\int_{\Omega}\!(\partial_{j}u_{i}\partial_{i}d\partial_{j}d-\frac{1}{2}\partial_{i}u_{i}|\partial_{j}d|^{2})\text{d}x
-\int_{\Omega}\!\partial_{j}u_{i}\partial_{i}d_{t}\partial_{j}d\text{d}x-\int_{\Omega}\!\partial_{j}u_{i}\partial_{i}d\partial_{j}d_{t}\text{d}x\nonumber\\
&+\int_{\Omega}\partial_{i}u_{i}\partial_{j}d_{t}\partial_{j}d\text{d}x-\frac{d}{dt}\int_{\Omega}(\frac{\partial_{i}u_{i}
|d|^{4}}{4}-\frac{\partial_{i}u_{i}|d|^{2}}{2})\text{d}x+\int_{\Omega}\partial_{i}u_{i}(|d|^{3}d_{t}-|d|d_{t})\text{d}x\}\nonumber\\
\leq
&\lambda\sum_{i,j=1}^{3}\frac{d}{dt}\int_{\Omega}(\partial_{j}u_{i}\partial_{i}d\partial_{j}d
-\frac{1}{2}\partial_{i}u_{i}|\partial_{j}d|^{2}-\frac{1}{4}\partial_{i}u_{i}
|d|^{4}+\frac{1}{2}\partial_{i}u_{i}|d|^{2})\text{d}x\nonumber\\
&+C\|\nabla u\|_{L^{2}}\|\nabla d\|_{L^{\infty}}\|\nabla
d_{t}\|_{L^{2}}+C\|\nabla
u\|_{L^{2}}\|d_{t}\|_{L^{2}}(\|d\|_{L^{\infty}}^{2}+1)\|d\|_{L^{\infty}}\nonumber\\
\leq
&\lambda\sum_{i,j=1}^{3}\frac{d}{dt}\int_{\Omega}(\partial_{j}u_{i}\partial_{i}d\partial_{j}d
-\frac{1}{2}\partial_{i}u_{i}|\partial_{j}d|^{2}-\frac{1}{4}\partial_{i}u_{i}
|d|^{4}+\frac{1}{2}\partial_{i}u_{i}|d|^{2})\text{d}x\nonumber\\
&+C\varepsilon^{-1}\|\nabla u\|_{L^{2}}^{2}+\varepsilon\|\nabla
d_{t}\|_{L^{2}}^{2}+C\|\nabla
u\|_{L^{2}}^{2}+C\|d_{t}\|_{L^{2}}^{2},
\end{align}
where we have used estimate \eqref{eq2.10} in the last inequality.
Inserting \eqref{eq2.19} and \eqref{eq2.20} into \eqref{eq2.17}, and
integrating over $[0,T]$ give that
\begin{align}\label{eq2.21}
&\|\nabla u\|_{L^{2}}^{2}+\int_{0}^{T}\int_{\Omega}\rho
|\dot{u}|^{2}\text{d}x\text{d}t\nonumber\\
\leq & C+\!C\!\int_{0}^{T}\!\!\int_{\Omega}\!u\cdot\nabla u\Delta
u\text{d}x\text{d}t+C\!\int_{\Omega}\!P(\rho)\operatorname{div}u\text{d}x(T)+C\!\int_{0}^{T}\!\!\int_{\Omega}\!P'(\rho)(\nabla
\rho\cdot
u+\rho\operatorname{div}u)\operatorname{div}u\text{d}x\text{d}t\nonumber\\
&+C\sum_{i,j=1}^{3}\int_{\Omega}(\partial_{j}u_{i}\partial_{i}d\partial_{j}d
-\frac{1}{2}\partial_{i}u_{i}|\partial_{j}d|^{2}-\frac{1}{4}\partial_{i}u_{i}
|d|^{4}+\frac{1}{2}\partial_{i}u_{i}|d|^{2})\text{d}x(T)\nonumber\\
&+\varepsilon\int_{0}^{T}\|\nabla d_{t}\|_{L^{2}}^{2}\text{d}t
+C\varepsilon^{-1}\int_{0}^{T}\|\nabla
u\|_{L^{2}}^{2}\text{d}t+C\int_{0}^{T}
(\|\nabla u\|_{L^{2}}^{2}+\|d_{t}\|_{L^{2}}^{2})\text{d}t\nonumber\\
&+\int_{0}^{T}\int_{\Omega}|u||\nabla u||\nabla
P|\text{d}x\text{d}t+C\int_{0}^{T}\int_{\Omega}|u||\nabla u| |\nabla
d|(|\Delta d|+|f(d)|)\text{d}x\text{d}t\nonumber\\
\leq & C+\varepsilon\int_{0}^{T}\|\nabla
d_{t}\|_{L^{2}}^{2}\text{d}t+C\int_{0}^{T}\int_{\Omega}u\cdot\nabla
u\Delta
u\text{d}x\text{d}t+C\int_{\Omega}P(\rho)\operatorname{div}u\text{d}x(T)\nonumber\\
&+C\int_{0}^{T}\int_{\Omega}P'(\rho)(\nabla \rho\cdot
u+\rho\operatorname{div}u)\operatorname{div}u\text{d}x\text{d}t\nonumber\\
&+C\sum_{i,j=1}^{3}\int_{\Omega}(\partial_{j}u_{i}\partial_{i}d\partial_{j}d
-\frac{1}{2}\partial_{i}u_{i}|\partial_{j}d|^{2}-\frac{1}{4}\partial_{i}u_{i}
|d|^{4}+\frac{1}{2}\partial_{i}u_{i}|d|^{2})\text{d}x(T)\nonumber\\
&+\int_{0}^{T}\int_{\Omega}|u||\nabla u||\nabla
P|\text{d}x\text{d}t+C\int_{0}^{T}\int_{\Omega}|u||\nabla u| |\nabla
d|(|\Delta d|+|f(d)|)\text{d}x\text{d}t,
\end{align}
where we have used the estimate \eqref{eq2.6} and \eqref{eq2.11}. To
estimate the terms on the right side of \eqref{eq2.21}, by using
Lemma \ref{lem2.1}, the estimates \eqref{eq2.6}, \eqref{eq2.10} and
\eqref{eq2.11}, we get
\begin{align}
    \label{eq2.22}
&\int_{0}^{T}\int_{\Omega}u\cdot\nabla u\Delta
u\text{d}x\text{d}t=\sum_{i,j=1}^{3}\int_{0}^{T}\int_{\Omega}
(-\partial_{j}u_{i}\partial_{i}u\partial_{j}u-u_{i}\partial_{i}\partial_{j}u\partial_{j}u)\text{d}x\text{d}t\nonumber\\
&\quad\quad\quad\quad\quad\quad\quad\quad
=\sum_{i,j=1}^{3}\int_{0}^{T}\int_{\Omega}
(-\partial_{j}u_{i}\partial_{i}u\partial_{j}u+\frac{1}{2}\partial_{i}u_{i}|\partial_{j}u|^{2})\text{d}x\text{d}t\nonumber\\
&\quad\quad\quad\quad\quad\quad\quad\quad \leq C\int_{0}^{T}\|\nabla u\|_{L^{3}}^{3}\text{d}t\leq C
\int_{0}^{T}\|\nabla u\|_{L^{\infty}}\|\nabla u\|_{L^{2}}^{2}\text{d}t\\
    \label{eq2.23}
&\int_{\Omega}P(\rho)\operatorname{div}u\text{d}x(T)\leq
\frac{1}{4}\|\nabla
u\|_{L^{2}}^{2}+C\int_{\Omega}|P(\rho)|^{2}\text{d}x\leq\frac{1}{4}\|\nabla
u\|_{L^{2}}^{2}+C;\\
     \label{eq2.24}
&\int_{0}^{T}\int_{\Omega}P'(\rho)(\nabla \rho\cdot
u)\operatorname{div} u\text{d}x\text{d}t\leq
C\int_{0}^{T}\|\nabla\rho\|_{L^{2}}\|u\|_{L^{2}}\|\operatorname{div}u\|_{L^{\infty}}\text{d}t\nonumber\\
&\quad\quad\quad\quad\quad\quad\quad\quad \leq C\int_{0}^{T}\|\nabla
\rho\|_{L^{2}}^{2}\|\nabla
u\|_{L^{\infty}}\text{d}t+C\int_{0}^{T}\|u\|_{L^{2}}^{2}\|\nabla
u\|_{L^{\infty}}\text{d}t\nonumber\\
&\quad\quad\quad\quad\quad\quad\quad\quad \leq C\int_{0}^{T}\|\nabla
\rho\|_{L^{2}}^{2}\|\nabla
u\|_{L^{\infty}}\text{d}t+C\int_{0}^{T}\|\nabla u\|_{L^{\infty}}\text{d}t\nonumber\\
&\quad\quad\quad\quad\quad\quad\quad\quad \leq C\int_{0}^{T}\|\nabla
\rho\|_{L^{2}}^{2}\|\nabla
u\|_{L^{\infty}}\text{d}t+C;\\
      \label{eq2.25}
&\int_{0}^{T}\int_{\Omega}P'(\rho)\rho |\operatorname{div}
u|^{2}\text{d}x\text{d}\leq
C\int_{0}^{T}\|\rho\|_{L^{\infty}}\|\nabla u\|_{L^{2}}\text{d}t\leq
C\int_{0}^{T}\|\nabla
u\|_{L^{2}}^{2}\text{d}t\leq C;\\
      \label{eq2.26}
&\sum_{i,j=1}^{3}\int_{\Omega}(\partial_{j}u_{i}\partial_{i}d\partial_{j}d
-\frac{1}{2}\partial_{i}u_{i}|\partial_{j}d|^{2}-\frac{1}{4}\partial_{i}u_{i}
|d|^{4}+\frac{1}{2}\partial_{i}u_{i}|d|^{2})\text{d}x(T)\nonumber\\
&\quad\quad\quad\quad\quad\quad\quad\quad \leq C(\|\nabla
u\|_{L^{2}}\|\nabla d\|_{L^{2}}
\|\nabla d\|_{L^{\infty}}+\|\nabla u\|_{L^{2}}\|d\|_{L^{2}}(\|d\|_{L^{\infty}}^{3}+\|d\|_{L^{\infty}}))\nonumber\\
&\quad\quad\quad\quad\quad\quad\quad\quad \leq \frac{1}{4}\|\nabla
u\|_{L^{2}}^{2}+C\|\nabla d\|_{L^{2}}^{2}\|\nabla
d\|_{L^{\infty}}^{2}+C\|d\|_{L^{2}}^{2}(\|d\|_{L^{\infty}}^{6}+\|d\|_{L^{\infty}}^{2})\nonumber\\
&\quad\quad\quad\quad\quad\quad\quad\quad \leq \frac{1}{4}\|\nabla
u\|_{L^{2}}^{2}+C\\
      \label{eq2.27}
&\int_{0}^{T}\int_{\Omega}|u||\nabla u||\nabla
P|\text{d}x\text{d}t\leq C\int_{0}^{T}\int_{\Omega}|u||\nabla
u||\nabla\rho|\text{d}x\text{d}t\nonumber\\
&\quad\quad\quad\quad\quad\quad \leq
C\int_{0}^{T}\|u\|_{L^{6}}\|\nabla u\|_{L^{3}}\|\nabla \rho\|_{L^{2}}\text{d}t
\leq C\int_{0}^{T}\|\nabla u\|_{L^{2}}^{\frac{5}{3}}\|\nabla u\|_{L^{\infty}}^{\frac{1}{3}}\|\nabla \rho\|_{L^{2}}\text{d}t\nonumber\\
&\quad\quad\quad\quad\quad\quad \leq C \int_{0}^{T}(\|\nabla
u\|_{L^{2}}^{2}\|\nabla\rho\|_{L^{2}}^{2}+
\|\nabla u\|_{L^{\infty}}^{\frac{2}{3}}\|\nabla u\|_{L^{2}}^{\frac{4}{3}})\text{d}t\nonumber\\
&\quad\quad\quad\quad\quad\quad \leq C\int_{0}^{T}(\|\nabla
u\|_{L^{2}}^{2}\|\nabla\rho\|_{L^{2}}^{2}+
\|\nabla u\|_{L^{\infty}}\|\nabla u\|_{L^{2}}^{2}+C)\text{d}t;\\
      \label{eq2.28}
&\int_{0}^{T}\int_{\Omega}|u||\nabla u| |\nabla d|(|\Delta
d|+|f(d)|)\text{d}x\text{d}t\leq C\int_{0}^{T}(\|u\|_{L^{6}}\|\nabla
u\|_{L^{6}}\|\nabla d\|_{L^{6}}\|\Delta d\|_{L^{2}}\nonumber\\
&\quad\quad\quad\quad\quad\quad+\|u\|_{L^{6}}\|\nabla u\|_{L^{2}}\|\nabla d\|_{L^{3}}(\|d\|_{L^{\infty}}^{3}+\|d\|_{L^{\infty}}))\text{d}t\nonumber\\
&\quad\quad\quad\quad\quad\quad \leq C\int_{0}^{T}(\|\nabla
u\|_{L^{2}}\|\nabla^{2} u\|_{L^{2}}\|\Delta d\|_{L^{2}}+\|\nabla
u\|_{L^{2}}^{2})\text{d}t\nonumber\\
&\quad\quad\quad\quad\quad\quad \leq \int_{0}^{T}(\varepsilon
\|\nabla^{2}u\|_{L^{2}}^{2}+C\varepsilon^{-1}\|\nabla
u\|_{L^{2}}^{2}\|\Delta d\|_{L^{2}}^{2})\text{d}t+C.
\end{align}
By using the stand elliptic regularity result to \eqref{eq2.16}, we
have
\begin{align}\label{eq2.29}
\|\nabla^{2}u\|_{L^{2}}^{2}\leq& \|\Delta u\|_{L^{2}}^{2}+\|\nabla
u\|_{L^{2}}^{2}\leq C(\|\nabla u\|_{L^{2}}^{2}+\|\rho
\dot{u}\|_{L^{2}}^{2}+\|\nabla P\|_{L^{2}}^{2}+\|(\nabla
d)^{T}(\Delta d-f(d))\|_{L^{2}}^{2})\nonumber\\
\leq& C(\|\nabla
u\|_{L^{2}}^{2}+\|\rho\|_{L^{\infty}}^{\frac{1}{2}}\|\rho^{\frac{1}{2}}
\dot{u}\|_{L^{2}}^{2}+\|\nabla \rho\|_{L^{2}}^{2}+\|\nabla
d\|_{L^{\infty}}^{2}(\|\Delta
d\|_{L^{2}}^{2}+\|f(d)\|_{L^{2}}^{2}))\nonumber\\
\leq& C(\|\nabla u\|_{L^{2}}^{2}+\|\rho^{\frac{1}{2}}
\dot{u}\|_{L^{2}}^{2}+\|\nabla \rho\|_{L^{2}}^{2}+\|\Delta
d\|_{L^{2}}^{2}+1).
\end{align}
Combining estimates \eqref{eq2.21}--\eqref{eq2.29} and taking
$\varepsilon$ small enough, we can get
\begin{align}\label{eq2.30}
&\|\nabla u\|_{L^{2}}^{2}+\int_{0}^{T}\int_{\Omega}\rho
|\dot{u}|^{2}\text{d}x\text{d}t\nonumber\\
\leq & C+\varepsilon\!\int_{0}^{T}\!\!\|\nabla
d_{t}\|_{L^{2}}^{2}\text{d}t+C\!\int_{0}^{T}\!\!(\|\nabla
\rho\|_{L^{2}}^{2}+\!\|\nabla u\|_{L^{2}}^{2}+\!\|\Delta
d\|_{L^{2}}^{2})(\|\nabla u\|_{L^{\infty}}+\!\|\nabla
u\|_{L^{2}}^{2}+1)\text{d}t.
\end{align}

To estimate the orientation parameter $d$, by the standard elliptic
regularity result to the liquid crystal equation $\eqref{NLC}_{3}$,
one obtains that
\begin{align}\label{eq2.31}
\|\nabla^{3} d\|_{L^{2}}\leq& C(\|\nabla
d_{t}\|_{L^{2}}+\|\nabla(u\cdot \nabla d)\|_{L^{2}}+\|\nabla
f(d)\|_{L^{2}}+\|d_{0}\|_{H^{3}})\nonumber\\
\leq & C(\|\nabla d_{t}\|_{L^{2}}+\|\nabla u\|_{L^{2}}\|\nabla
d\|_{L^{\infty}}+\||u||\nabla^{2}d|\|_{L^{2}}+\|\nabla d\|_{L^{2}}
(\|d\|_{L^{\infty}}^{2}\!+\!\|d\|_{L^{\infty}}\!)+\|d_{0}\|_{H^{3}})\nonumber\\
\leq & C(\|\nabla d_{t}\|_{L^{2}}+\|\nabla
u\|_{L^{2}}+\||u||\nabla^{2}d|\|_{L^{2}}+C)
\end{align}
Multiplying the liquid crystal equation $\eqref{NLC}_{3}$ by $\Delta
d_{t}$, and integrating over $\Omega$, then we have
\begin{align}\label{eq2.32}
&\frac{d}{dt}\int_{\Omega}\nu|\Delta
d|^{2}\text{d}x+\int_{\Omega}|\nabla d_{t}|^{2}\text{d}x\nonumber\\
=&\int_{\Omega} u\cdot \nabla d\Delta
d_{t}\text{d}x+\nu\int_{\Omega}(|d|^{2}-1)d\Delta
d_{t}\text{d}x\nonumber\\
=&\sum_{i,j=1}^{3}\int_{\Omega}u_{i}\partial_{i}d\partial_{j}^{2}d_{t}\text{d}x-\nu\int_{\Omega}\nabla(|d|^{2}d)\nabla
d_{t}\text{d}x+\nu\int_{\Omega}\nabla d\nabla d_{t}\text{d}x\nonumber\\
=&-\sum_{i,j=1}^{3}\int_{\Omega}\partial_{j}u_{i}\partial_{i}d\partial_{j}d_{t}\text{d}x-\sum_{i,j=1}^{3}\int_{\Omega}
u_{i}\partial_{i}\partial_{j}d\partial_{j}d_{t}\text{d}x-\nu\int_{\Omega}\nabla(|d|^{2}d)\nabla
d_{t}\text{d}x+\nu\int_{\Omega}\nabla d\nabla d_{t}\text{d}x\nonumber\\
\leq & C(\|\nabla u\|_{L^{2}}\!\|\nabla d\|_{L^{\infty}}\!\|\nabla
d_{t}\|_{L^{2}}\!\nonumber\\
&+\|u\nabla^{2} d\|_{L^{2}}\!\|\nabla
d_{t}\|_{L^{2}}\!+\|d\|_{L^{\infty}}^{2}\!\|\nabla
d\|_{L^{2}}\!\|\nabla d_{t}\|_{L^{2}}\!+\|\nabla
d\|_{L^{2}}\!\|\nabla
d_{t}\|_{L^{2}})\nonumber\\
\leq& \varepsilon\|\nabla
d_{t}\|_{L^{2}}^{2}+C\varepsilon^{-1}(\|\nabla
u\|_{L^{2}}^{2}+\int_{\Omega}|u|^{2}|\nabla^{2} d|^{2}\text{d}x+1),
\end{align}
where we have used the H\"{o}lder inequality and estimates
\eqref{eq2.10}. For the term $\int_{\Omega}|u|^{2}|\nabla^{2}
d|^{2}\text{d}x$, applying estimate \eqref{eq2.31}, we have for
$\eta>0$
\begin{align*}
\int_{\Omega}|u|^{2}|\nabla^{2} d|^{2}\text{d}x&\leq C
\|u\|_{L^{6}}^{2}\|\nabla^{2} d\|_{L^{3}}^{2}\leq C\|\nabla
u\|_{L^{2}}^{2}\|\nabla
d\|_{L^{6}}\|\nabla^{3}d\|_{L^{2}}\nonumber\\
\leq& \eta\|\nabla^{3}d\|_{L^{2}}^{2}+C\eta^{-1}\|\nabla
u\|_{L^{2}}^{4}\nonumber\\
\leq& \eta\|\nabla
d_{t}\|_{L^{2}}^{2}+\eta\int_{\Omega}|u|^{2}|\nabla^{2}d|^{2}\text{d}x+\eta\|\nabla
u\|_{L^{2}}^{2}+C\eta^{-1}(\|\nabla u\|_{L^{2}}^{4}+C).
\end{align*}
Hence, taking $\eta$ small enough
\begin{align}\label{eq2.33}
\int_{\Omega}|u|^{2}|\nabla^{2} d|^{2}\text{d}x\leq& 2\eta\|\nabla
d_{t}\|_{L^{2}}^{2}+C\eta^{-1}(\|\nabla u\|_{L^{2}}^{2}(\|\nabla
u\|_{L^{2}}^{2}+1)+C).
\end{align}
Inserting \eqref{eq2.33} into \eqref{eq2.32}, taking $\varepsilon$,
$\eta$ small enough and integrating above inequality over $(0;T]$
ensure that
\begin{align}\label{eq2.34}
\|\Delta d\|_{L^{2}}^{2}+\int_{0}^{T}\int_{\Omega}|\nabla
d_{t}|^{2}\text{d}x\text{d}t \leq C\int_{0}^{T}\|\nabla
u\|_{L^{2}}^{2}(\|\nabla u\|_{L^{2}}^{2}+1)\text{d}t+C.
\end{align}

Now, we will estimate the density $\rho$. Applying the operator
$\nabla$ to the mass conservation equation $\eqref{NLC}_{1}$, then
multiplying it by $\nabla \rho$ and integrating over $\Omega$ yield
\begin{align*}
\frac{d}{dt}\|\nabla\rho\|_{L^{2}}^{2}=&-\int_{\Omega}|\nabla
\rho|^{2}\operatorname{div}u\text{d}x-2\int_{\Omega}\rho\nabla\rho\nabla\operatorname{div}u\text{d}x
-2\int_{\Omega}(\nabla \rho\cdot\nabla u)\nabla
\rho\text{d}x\nonumber\\
\leq& C\|\nabla\rho\|_{L^{2}}^{2}\|\nabla u\|_{L^{\infty}}+C\|\nabla
\rho\|_{L^{2}}\|\nabla\operatorname{div}u\|_{L^{2}}\nonumber\\
\leq& \varepsilon \|\nabla^{2}
u\|_{L^{2}}^{2}+C\varepsilon^{-1}\|\nabla\rho\|_{L^{2}}^{2}(\|\nabla
u\|_{L^{\infty}}+1) \nonumber\\
\leq& \varepsilon (\|\rho^{\frac{1}{2}}
\dot{u}\|_{L^{2}}^{2}+\|\nabla
u\|_{L^{2}}^{2}+\|\nabla\rho\|_{L^{2}}^{2}+\|\Delta
d\|_{L^{2}}^{2}+1)+C\varepsilon^{-1}\|\nabla\rho\|_{L^{2}}^{2}(\|\nabla
u\|_{L^{\infty}}+1)\nonumber\\
\leq& \varepsilon \|\rho^{\frac{1}{2}}
\dot{u}\|_{L^{2}}^{2}+C\varepsilon^{-1}\|\nabla\rho\|_{L^{2}}^{2}(\|\nabla
u\|_{L^{\infty}}+1)+C(\|\nabla u\|_{L^{2}}^{2}+\|\Delta
d\|_{L^{2}}^{2}),
\end{align*}
where we have used the estimate \eqref{eq2.29} in the above
inequality. Integrating the above estimate over $(0,T]$ gives that
\begin{align}\label{eq2.35}
\|\nabla \rho\|_{L^{2}}^{2}\leq
\varepsilon\int_{0}^{T}\|\rho^{\frac{1}{2}}\dot{u}\|_{L^{2}}^{2}\text{d}t+\int_{0}^{T}(C\varepsilon^{-1}\|\nabla\rho\|_{L^{2}}^{2}(\|\nabla
u\|_{L^{\infty}}+1)+C(\|\nabla u\|_{L^{2}}^{2}+\|\Delta
d\|_{L^{2}}^{2}))\text{d}t
\end{align}
Combining estimates \eqref{eq2.30}, \eqref{eq2.34} and
\eqref{eq2.35}, and taking $\varepsilon$ small enough, one obtains
that
\begin{align}\label{eq2.36}
\|\nabla u\|_{L^{2}}^{2}&+\|\nabla\rho\|_{L^{2}}^{2}+\|\Delta
d\|_{L^{2}}^{2}+\int_{0}^{T}\int_{\Omega}(\rho|\dot{u}|^{2}+|\nabla
d_{t}|^{2})\text{d}x\text{d}t\nonumber\\
\leq &C+C\int_{0}^{T}(\|\nabla \rho\|_{L^{2}}^{2}+\|\nabla
u\|_{L^{2}}^{2}+\|\Delta d\|_{L^{2}}^{2})(\|\nabla
u\|_{L^{\infty}}+\|\nabla u\|_{L^{2}}^{2}+1)\text{d}t.
\end{align}
Since the energy estimate \eqref{eq2.6} implies that
$\int_{0}^{T}\|\nabla u\|_{L^{2}}^{2}\text{d}t\leq C$. By using the
Gronwall's inequality, the elliptic regularity result
$\|\nabla^{2}d\|_{L^{2}}\leq C(\|\Delta
d\|_{L^{2}}+\|d_{0}\|_{H^{2}})$ and noticing that the assumption
\eqref{eq2.1}, we deduce that the inequality \eqref{eq2.14} holds.

To prove the estimate \eqref{eq2.15}, by using the standard elliptic
regularity result on $\eqref{NLC}_{3}$, we have
\begin{align}\label{eq2.37}
\|\nabla^{2} d\|_{L^{3}}^{2}\leq& C(\|
d_{t}\|_{L^{3}}^{2}+\|u\cdot\nabla
d\|_{L^{3}}^{2}+\|f(d)\|_{L^{3}}^{2}+\|d_{0}\|_{H^{3}}^{2})\nonumber\\
\leq & C(\|d_{t}\|_{L^{2}}\|\nabla
d_{t}\|_{L^{2}}+\|u\|_{L^{6}}^{2}\|\nabla
d\|_{L^{6}}^{2}+\|d\|_{L^{3}}^{2}(\|d\|_{L^{\infty}}^{4}+\|d\|_{L^{2}}^{2})+C)\nonumber\\
\leq& C(\|d_{t}\|_{L^{2}}^{2}+\|\nabla d_{t}\|_{L^{2}}^{2}+\|\nabla
u\|_{L^{2}}^{2}+C)\nonumber\\
\leq& C(\|d_{t}\|_{L^{2}}^{2}+\|\nabla d_{t}\|_{L^{2}}^{2}+C),
\end{align}
where we have used the estimate \eqref{eq2.14} in the last
inequality. Then by using the estimates \eqref{eq2.10},
\eqref{eq2.11}, \eqref{eq2.14} and the above inequality, we have
\begin{align*}
&\int_{0}^{T}\|\nabla d\|_{H^{2}}^{2}\text{d}t\leq
\int_{0}^{T}(\|\nabla^{3} d\|_{L^{2}}^{2}+\|\nabla
d\|_{L^{2}}^{2})\text{d}t\nonumber\\
\leq& C\int_{0}^{T}(\|\nabla d_{t}\|_{L^{2}}^{2}+\|\nabla
(u\cdot\nabla d)\|_{L^{2}}^{2}+\|\nabla
f(d)\|_{L^{2}}^{2}+C)\text{d}t\nonumber\\
\leq & C\int_{0}^{T}(\|\nabla d_{t}\|_{L^{2}}^{2}+\|\nabla
u\|_{L^{2}}^{2}\|\nabla
d\|_{L^{\infty}}^{2}+\|u\|_{L^{6}}^{2}\|\nabla^{2}
d\|_{L^{3}}^{2}+\|\nabla
d\|_{L^{2}}^{2}(\|d\|_{L^{\infty}}^{4}+\|d\|_{L^{\infty}}^{2})+C)\text{d}t\nonumber\\
\leq& C\int_{0}^{T}(\|\nabla d_{t}\|_{L^{2}}^{2}+\|\nabla^{2}
d\|_{L^{2}}^{2}+C)\text{d}t\nonumber\\
\leq& C\int_{0}^{T}(\|\nabla d_{t}\|_{L^{2}}^{2}+\|
d_{t}\|_{L^{2}}^{2}+C)\text{d}t\leq C.
\end{align*}
This completes the proof of Lemma \ref{lem2.4}.
\end{proof}

\begin{lemma}\label{lem2.5}
Under the assumption \eqref{eq2.1}, it holds that for $0\leq T<
T^{*}$
\begin{align}\label{eq2.38}
&\sup_{0\leq t\leq
T}(\|\rho^{\frac{1}{2}}u_{t}\|_{L^{2}}^{2}+\|\nabla
d_{t}\|_{L^{2}}^{2})+\int_{0}^{T}(\|\nabla
u_{t}\|_{L^{2}}^{2}+\|(\Delta
d-f(d))_{t}\|_{L^{2}}^{2})\text{d}t\leq C.
\end{align}
\end{lemma}

\begin{proof}
Differentiating the momentum equation $\eqref{NLC}_{2}$ with respect
to time, multiplying the resulting equation by $u_{t}$, integrating
it over $\Omega $ and making use of the mass conservation equation
$\eqref{NLC}_{1}$, one obtains that
\begin{align}\label{eq2.39}
&\frac{1}{2}\frac{d}{dt}\int_{\Omega}\rho|u_{t}|^{2}\text{d}x+\mu\int_{\Omega}|\nabla
u_{t}|^{2}\text{d}x-\int_{\Omega}P_{t}\operatorname{div}u_{t}\text{d}x\nonumber\\
=&-\int_{\Omega}\rho
u\cdot\nabla(\frac{|u_{t}|^{2}}{2}+(u\cdot\nabla )u u_{t})+\rho
(u_{t}\cdot\nabla)u
u_{t}\text{d}x-\lambda\int_{\Omega}(u_{t}\cdot\nabla )d_{t}(\Delta
d-f(d))\text{d}t\nonumber\\
&-\lambda\int_{\Omega}(u_{t}\cdot\nabla)d(\Delta
d-f(d))_{t}\text{d}x\nonumber\\
=&\int_{\Omega}\nabla\rho \cdot
u\frac{|u_{t}|^{2}}{2}+\rho\operatorname{div}
u\frac{|u_{t}|^{2}}{2}-\rho u\cdot \nabla((u\cdot\nabla )u
u_{t})-\rho (u_{t}\cdot\nabla)u
u_{t}\text{d}x\nonumber\\
&-\lambda\int_{\Omega}(u_{t}\cdot\nabla )d_{t}(\Delta
d-f(d))\text{d}t -\lambda\int_{\Omega}(u_{t}\cdot\nabla)d(\Delta
d-f(d))_{t}\text{d}x.
\end{align}

Differentiating the liquid crystal equation $\eqref{NLC}_{3}$ with
respect to time gives
\begin{align*}
(u_{t}\cdot\nabla)d=\nu (\Delta d-f(d))_{t}-d_{tt}-(u\cdot\nabla
)d_{t}.
\end{align*}
Multiplying the above equality with $(\Delta d-f(d))_{t}$ and
integrating over $\Omega$, one obtains the equality
\begin{align}\label{eq2.40}
&\int_{\Omega}(u_{t}\cdot\nabla)d(\Delta
d-f(d))_{t}\text{d}x\nonumber\\
=&\int_{\Omega} (\nu|(\Delta d-f(d))_{t}|^{2}-d_{tt}\Delta d_{t}+
d_{tt}f(d)_{t}-(u\cdot\nabla)d_{t}(\Delta
d-f(d))_{t})\text{d}x\nonumber\\
=&\int_{\Omega}\nu|(\Delta
d-f(d))_{t}|^{2}\text{d}x+\frac{1}{2}\frac{d}{dt}\|\nabla
d_{t}\|_{L^{2}}^{2}-\int_{\Omega}((u_{t}\cdot\nabla)d) f(d)_{t}\text{d}x\nonumber\\
&+\int_{\Omega}\nu f(d)_{t}(\Delta
d-f(d))_{t}\text{d}x-\int_{\Omega} ((u\cdot \nabla) d_{t})\Delta
d_{t}\text{d}x\nonumber\\
=&\int_{\Omega}\nu|(\Delta
d-f(d))_{t}|^{2}\text{d}x+\frac{1}{2}\frac{d}{dt}\|\nabla
d_{t}\|_{L^{2}}^{2}-\int_{\Omega}((u_{t}\cdot\nabla)d) f(d)_{t}\text{d}x\nonumber\\
&+\int_{\Omega}\nu f(d)_{t}(\Delta d-f(d))_{t}
\text{d}x+\int_{\Omega} ((\nabla u\cdot \nabla) d_{t}\nabla
d_{t}-\frac{1}{2}\operatorname{div}u|\nabla d_{t}|^{2})\text{d}x,
\end{align}
where we have used the fact
\begin{align*}
-\int_{\Omega} ((u\cdot \nabla) d_{t})\Delta
d_{t}\text{d}x=&-\sum_{i,j=1}^{3}\int_{\Omega}
u_{i}\partial_{i}d_{t}\partial_{jj}d_{t}\text{d}x\\
=&\sum_{i,j=1}^{3}\int_{\Omega}(\partial_{j}u_{i}\partial_{i}d_{t}\partial_{j}d_{t}+
u_{i}\partial_{i}(\frac{|\partial_{j}d_{t}|^{2}}{2})\text{d}x\\
=&\sum_{i,j=1}^{3}(\int_{\Omega}(\partial_{j}u_{i}\partial_{i}d_{t}\partial_{j}d_{t}\text{d}x
-\frac{1}{2}\int_{\Omega}\partial_{i}u_{i}|\partial_{j}d_{t}|^{2}\text{d}x)\\
=&\int_{\Omega} ((\nabla u\cdot \nabla) d_{t}\nabla
d_{t}-\frac{1}{2}\operatorname{div}u|\nabla d_{t}|^{2})\text{d}x
\end{align*}
in the last equality.

From the equation \eqref{eq2.18}, we can derive
\begin{align}\label{eq2.41}
\int_{\Omega
}P_{t}\operatorname{div}u_{t}\text{d}x=-\int_{\Omega}P'(\rho)(\nabla
\rho\cdot
u+\rho\operatorname{div}u)\operatorname{div}u_{t}\text{d}x.
\end{align}
Inserting the equalities \eqref{eq2.40} and \eqref{eq2.41} into
\eqref{eq2.39} derives
\begin{align}\label{eq2.42}
&\frac{d}{dt}\int_{\Omega}(\frac{1}{2}\rho |u_{t}|^{2}+\lambda
|\nabla d_{t}|^{2})\text{d}x+\mu\|\nabla
u_{t}\|_{L^{2}}^{2}+\lambda\nu\|(\Delta
d-f(d))_{t}\|_{L^{2}}^{2}\nonumber\\
\leq&
C\!\!\int_{\Omega}\!(|\nabla\rho||u||u_{t}|^{2}\!+\!\rho|\operatorname{div}u||u_{t}|^{2}\!+\!\rho|u||u_{t}||\nabla
u|^{2}\!+\!\rho |u|^{2}|u_{t}||\nabla^{2}u|\!+\!\rho |u|^{2}|\nabla
u||\nabla
u_{t}|\!+\!\rho|u_{t}|^{2}|\nabla u|)\text{d}x\nonumber\\
&+C\int_{\Omega}(|(u_{t}\cdot\nabla)d f(d)_{t}|+|(\Delta
d-f(d))_{t}f(d)_{t}|+|(\nabla u\cdot\nabla )d_{t}\nabla
d_{d}|+|\operatorname{div}u||\nabla d_{t}|^{2})\text{d}x\nonumber\\
&+C\int_{\Omega}|(u_{t}\cdot\nabla)d_{t}(\Delta
d-f(d))|\text{d}x+C\int_{\Omega}|p'(\rho)||\nabla
\rho||u||\operatorname{div}u_{t}|+\rho|P'(\rho)||\operatorname{div}u||\operatorname{div}u_{t}|\text{d}x\nonumber\\
=& \sum_{j=1}^{13}J_{j}.
\end{align}
We will estimate $J_{j}$ term by term. In the following
calculations, we will make extensive use of Sobolev embedding,
H\"{o}lder inequality, Lemmas \ref{lem2.1}--\ref{lem2.4} and
 the estimate \eqref{eq2.5},
\begin{align*}
&J_{1}\leq C\|\nabla
\rho\|_{L^{2}}\|u_{t}\|_{L^{6}}^{2}\|u\|_{L^{6}}\leq C\|\nabla
u_{t}\|_{L^{2}}^{2}\|\nabla u\|_{L^{2}}\leq \varepsilon\|\nabla
u_{t}\|_{L^{2}}^{2}+C\varepsilon^{-1};\\
&J_{2}+J_{6}\leq C \|\nabla u\|_{L^{\infty}}\|\rho^{\frac{1}{2}}
u_{t}\|_{L^{2}}^{2};\\
&J_{7}\leq C\|u_{t}\|_{L^{2}}\|\nabla
d\|_{L^{3}}(\|d\|_{L^{\infty}}^{2}+1)\|d_{t}\|_{L^{6}}\\
&\quad\leq C\| \rho^{\frac{1}{2}}u_{t}\|_{L^{2}}\|\nabla
d_{t}\|_{L^{2}}\leq \varepsilon\|\rho^{\frac{1}{2}}
u_{t}\|_{L^{2}}^{2}+C\varepsilon^{-1} \|\nabla
d_{t}\|_{L^{2}}^{2};\\
&J_{8}\leq C\|(\Delta d-f(d))_{t}\|_{L^{2}}\|f(d)_{t}\|_{L^{2}}\leq
\varepsilon\|(\Delta
d-f(d))_{t}\|_{L^{2}}^{2}+C\varepsilon^{-1}\|d_{t}\|_{L^{2}}^{2};\\
&J_{9}+J_{10}=\int_{\Omega}|(\nabla u\cdot\nabla)d_{t}\nabla d_{t}|+|\operatorname{div}u||\nabla d_{t}|^{2}\text{d}x\\
&\quad\quad\quad \leq C\|\nabla u\|_{L^{\infty}}\|\nabla
d_{t}\|_{L^{2}}^{2}; \\
&J_{11}\leq C\|u_{t}\|_{L^{6}}\|\nabla d_{t}\|_{L^{2}}\|\Delta
d\|_{L^{3}}+C\|u_{t}\|_{L^{2}}\|\nabla
d_{t}\|_{L^{2}}(\|d\|_{L^{\infty}}^{2}+1)\|d\|_{L^{\infty}}\\
&\quad\leq C\|\nabla u_{t}\|_{L^{2}}\|\nabla
d_{t}\|_{L^{2}}\|d\|_{H^{2}}^{\frac{1}{2}}\|\nabla
d\|_{H^{2}}^{\frac{1}{2}}+C\|\rho^{\frac{1}{2}}u_{t}\|_{L{2}}\|\nabla
d_{t}\|_{L^{2}}\\
&\quad\leq \varepsilon\|\nabla
u_{t}\|_{L^{2}}^{2}+C\varepsilon^{-1}\|\nabla
d_{t}\|_{L^{2}}^{2}(\|\nabla
d\|_{H^{2}}^{2}+1)+\varepsilon\|\rho^{\frac{1}{2}}u_{t}\|_{L^{2}}^{2}+C\varepsilon^{-1}\|\nabla
d_{t}\|_{L^{2}}^{2}\\
&\quad\leq \varepsilon\|\nabla
u_{t}\|_{L^{2}}^{2}+\varepsilon\|\rho^{\frac{1}{2}}u_{t}\|_{L^{2}}^{2}+C\varepsilon^{-1}\|\nabla
d_{t}\|_{L^{2}}^{2}(\|\nabla d\|_{H^{2}}^{2}+1);
\end{align*}
and
\begin{align*}
J_{13}\leq C\|\rho\|_{L^{\infty}}\|\nabla u\|_{L^{2}}\|\nabla
u_{t}\|_{L^{2}}\leq \varepsilon\|\nabla
u_{t}\|_{L^{2}}^{2}+C\varepsilon^{-1}.
\end{align*}
To estimate the terms $J_{3},J_{4},J_{5}$ and $J_{12}$,  by using
the standard elliptic estimate on \eqref{eq2.16} and making use of
the liquid crystal equation $\eqref{NLC}_{3}$ yield that
\begin{align*}
\|u\|_{H^{2}}\leq &C(\|\rho u_{t}\|_{L^{2}}+\|\rho u\cdot\nabla
u\|_{L^{2}}+\|\nabla P\|_{L^{2}}+\|(\nabla d)^{T}(
\Delta d-f(d))\|_{L^{2}})\nonumber\\
\leq &C(\|\rho
u_{t}\|_{L^{2}}+\|\rho u\cdot\nabla u\|_{L^{2}}+\|\nabla
P\|_{L^{2}}+\|(\nabla d)^{T}(
d_{t}+(u\cdot\nabla)d)\|_{L^{2}})\nonumber\\
\leq
&C(\|\rho^{\frac{1}{2}}u_{t}\|_{L^{2}}+\|\rho\|_{L^{6}}\|u\|_{L^{6}}\|\nabla
u\|_{L^{6}}+\|\nabla \rho\|_{L^{2}}+\|(\nabla
d)^{T}(d_{t}+(u\cdot\nabla)d)\|_{L^{2}}\nonumber\\
\leq &C(\|\rho^{\frac{1}{2}}u_{t}\|_{L^{2}}+\sigma\|\nabla
u\|_{H^{1}}+\sigma^{-1}\|\nabla u\|_{L^{2}}+\|\nabla
\rho\|_{L^{2}}+\|\nabla d\|_{L^{3}}\|d_{t}\|_{L^{6}}+\|\nabla
d\|_{L^{2}}^{2}\|u\|_{L^{6}})\nonumber\\
\leq&C(\|\rho^{\frac{1}{2}}u_{t}\|_{L^{2}}+\sigma\|\nabla
u\|_{H^{1}}+\sigma^{-1}\|\nabla u\|_{L^{2}}+\|\nabla
\rho\|_{L^{2}}+\|\nabla d_{t}\|_{L^{2}}),\nonumber
\end{align*}
where we have used the estimates \eqref{eq2.6} and \eqref{eq2.10} in
the last inequality. Taking $\sigma$ small enough yields
\begin{align}\label{eq2.43}
\|u\|_{H^{2}}\leq &C(\|\rho^{\frac{1}{2}}u_{t}\|_{L^{2}}+\|\nabla
u\|_{L^{2}}+\|\nabla \rho\|_{L^{2}}+\|\nabla d_{t}\|_{L^{2}}).
\end{align}
Making use of estimates \eqref{eq2.14} and \eqref{eq2.43}, we can
estimate $J_{3},J_{4},J_{5}$ and  $J_{12}$ as
\begin{align*}
&J_{3}+J_{4}+J_{5}=\int_{\Omega}\rho |u||\nabla
u|^{2}|u_{t}|+\rho|u|^{2}|\nabla^{2} u||u_{t}|+\rho|u|^{2}|\nabla
u||\nabla u_{t}|\text{d}x\\
&\quad\quad\leq
C(\|\rho\|_{L^{\infty}}\|u\|_{L^{6}}\|u_{t}\|_{L^{6}}\|\nabla
u\|_{L^{2}}\|\nabla
u\|_{L^{6}}+\|\rho\|_{L^{\infty}}\|u_{t}\|_{L^{6}}\|u\|_{L^{6}}^{2}\|\nabla^{2}u\|_{L^{2}}\\
&\quad\quad\quad\quad+\|\rho\|_{L^{\infty}}\|u\|_{L^{6}}^{2}\|\nabla
u\|_{L^{6}}\|\nabla u_{t}\|_{L^{2}})\\
&\quad\quad\leq C\|\nabla u_{t}\|_{L^{2}}\|\nabla
u\|_{L^{2}}^{2}\|u\|_{H^{2}}\leq \varepsilon\|\nabla
u_{t}\|_{L^{2}}^{2}+C\varepsilon^{-1}\|u\|_{H^{2}}^{2}\\
&\quad\quad\leq\varepsilon\|\nabla
u_{t}\|_{L^{2}}^{2}+C\varepsilon^{-1}(\|\rho^{\frac{1}{2}}u_{t}\|_{L^{2}}^{2}+\|\nabla
d_{t}\|_{L^{2}}^{2}+\|\nabla u\|_{L^{2}}^{2}+\|\nabla
\rho\|_{L^{2}}^{2})\\
&\quad\quad\leq\varepsilon\|\nabla
u_{t}\|_{L^{2}}^{2}+C\varepsilon^{-1}(\|\rho^{\frac{1}{2}}u_{t}\|_{L^{2}}^{2}+\|\nabla
d_{t}\|_{L^{2}}^{2}+1);\\
&J_{12}\leq C\|\nabla \rho\|_{L^{2}}\|u\|_{L^{\infty}}\|\nabla
u_{t}\|_{L^{2}}\leq C\|\nabla u_{t}\|_{L^{2}}\|u\|_{H^{2}}\\
&\quad\leq \varepsilon\|\nabla
u_{t}\|_{L^{2}}^{2}+C\varepsilon^{-1}\|u\|_{H^{2}}^{2}\\
&\quad\leq \varepsilon\|\nabla
u_{t}\|_{L^{2}}^{2}+C\varepsilon^{-1}(\|\rho^{\frac{1}{2}}u_{t}\|_{L^{2}}^{2}+\|\nabla
d_{t}\|_{L^{2}}^{2}+\|\nabla u\|_{L^{2}}^{2}+\|\nabla
\rho\|_{L^{2}}^{2})\\
&\quad\leq \varepsilon\|\nabla
u_{t}\|_{L^{2}}^{2}+C\varepsilon^{-1}(\|\rho^{\frac{1}{2}}u_{t}\|_{L^{2}}^{2}+\|\nabla
d_{t}\|_{L^{2}}^{2}+1)
\end{align*}
Substituting all the estimates of $J_{j}$ into \eqref{eq2.42}, and
taking $\varepsilon$ small enough, we obtain
\begin{align}\label{eq2.44}
&\frac{d}{dt}\int_{\Omega}(\rho |u_{t}|^{2}+|\nabla
d_{t}|^{2})\text{d}x+\|\nabla u_{t}\|_{L^{2}}^{2}+\|(\Delta
d-f(d))_{t}\|_{L^{2}}^{2}\nonumber\\
\leq & C[\|\rho^{\frac{1}{2}}u_{t}\|_{L^{2}}^{2}(\|\nabla
u\|_{L^{\infty}}+C)+\|\nabla d_{t}\|_{L^{2}}^{2}(\|\nabla
u\|_{L^{\infty}}+\|\nabla d\|_{H^{2}}^{2}+C)+\|\nabla
d_{t}\|_{L^{2}}^{2}+1]\nonumber\\
\leq & C(\|\rho^{\frac{1}{2}}u_{t}\|_{L^{2}}^{2}+\|\nabla
d_{t}\|_{L^{2}}^{2})(\|\nabla u\|_{L^{\infty}}+\|\nabla
d\|_{H^{2}}^{2}+C)+C(\|\nabla d_{t}\|_{L^{2}}^{2}+1).
\end{align}
Applying the Gronwall's inequality to estimate \eqref{eq2.44}, we
deduce
\begin{align}\label{eq2.45}
&\sup_{0\leq t\leq T}\int_{\Omega }(\rho |u_{t}|^{2}+|\nabla
d_{t}|^{2})\text{d}x+\int_{0}^{T}\|\nabla
u_{t}\|_{L^{2}}^{2}+\|(\Delta
d-f(d))_{t}\|_{L^{2}}^{2}\text{d}t\nonumber\\
\leq& C\int_{0}^{T}(\|\nabla
d_{t}\|_{L^{2}}^{2}+1)\text{d}t\exp\{\int_{0}^{T}(\|\nabla
u\|_{L^{\infty}}+\|\nabla d\|_{H^{2}}^{2}+C)\text{d}t\}\leq C,
\end{align}
where we have used estimate \eqref{eq2.15} and the assumption
\eqref{eq2.1} in the last inequality. This completes the proof of
Lemma \ref{lem2.5}.
\end{proof}
\medskip

The following lemma gives the higher order norm estimates of $u$,
$d$ and $\rho$.

\begin{lemma}\label{lem2.6}
Under the assumption \eqref{eq2.1}, it holds that for $0\leq T<
T^{*}$
\begin{align}
     \label{eq2.46}
&\sup_{0\leq t\leq T}(\|u\|_{H^{2}}+\|\nabla d\|_{H^{2}})\leq C;\\
      \label{eq2.47}
& \sup_{0\leq t\leq T}\|\nabla
\rho\|_{L^{6}}+\int_{0}^{T}\|\nabla^{2} u\|_{L^{6}}^{2}\text{d}t\leq
C.
\end{align}
\end{lemma}

\begin{proof}
From estimates \eqref{eq2.14}, \eqref{eq2.38} and \eqref{eq2.43}, we
have
\begin{align}\label{eq2.48}
\|u\|_{H^{2}}\leq C (\|\rho^{\frac{1}{2}}u_{t}\|_{L^{2}}+\|\nabla
u\|_{L^{2}}+\|\nabla \rho\|_{L^{2}}+\|\nabla d_{t}\|_{L^{2}})\leq C.
\end{align}
By using estimates \eqref{eq2.31} and \eqref{eq2.33}, we have
\begin{align}\label{eq2.49}
\|\nabla d\|_{H^{2}}\leq& C(\|\nabla^{3}d\|_{L^{2}}+\|\nabla
d\|_{L^{2}})\nonumber\\
\leq &C(\|\nabla d_{t}\|_{L^{2}}+\|\nabla
u\|_{L^{2}}+\||u||\nabla^{2}d|\|_{L^{2}}+\|\nabla d\|_{L^{2}})\nonumber\\
\leq& C(\|\nabla d_{t}\|_{L^{2}}+\|\nabla u\|_{L^{2}}+\|\nabla
u\|_{L^{2}}^{2}+\|\nabla d\|_{L^{2}}+C)\leq C,
\end{align}
where we have used the estimates \eqref{eq2.10}, \eqref{eq2.14} and
\eqref{eq2.38} in the last inequality. Combining the estimates
\eqref{eq2.48} and \eqref{eq2.49} above gives the estimate
\eqref{eq2.46}.

Applying the operator $\nabla$ to the mass conservation equation
$\eqref{NLC}_{1}$, then multiplying the resulting equation by
$6|\nabla \rho|^{4}\nabla\rho$ and integrating over $\Omega$ give
\begin{align*}
\frac{d}{dt}\|\nabla
\rho\|_{L^{6}}^{6}=&\!-6\!\int_{\Omega}\!|\nabla \rho|^{6}\nabla
u\text{d}x-\!\int_{\Omega }\!\nabla(|\nabla \rho|^{6})\cdot
u\text{d}x-6\!\int_{\Omega}\!|\nabla\rho|^{6}\operatorname{div}u\text{d}x
-6\!\int_{\Omega}\!\rho|\nabla\rho|^{4}\nabla\rho\nabla\operatorname{div}u\text{d}x\nonumber\\
\leq&C\|\nabla u\|_{L^{\infty}}\|\nabla \rho\|_{L^{6}}^{6}+C\|\nabla
\operatorname{div}u\|_{L^{6}}\|\nabla\rho\|_{L^{6}}^{5},
\end{align*}
that is
\begin{align}\label{eq2.50}
\frac{d}{dt}\|\nabla\rho\|_{L^{6}}\leq C\|\nabla
u\|_{L^{\infty}}\|\nabla\rho\|_{L^{6}}+C\|\nabla\operatorname{div}
u\|_{L^{6}}.
\end{align}
By using the Gronwall's inequality to the above estimate gives
\begin{align}\label{eq2.51}
\|\nabla \rho\|_{L^{6}}\leq&
(\|\rho_{0}\|_{W^{1,6}}+C\int_{0}^{T}\|\nabla\operatorname{div}u\|_{L^{6}}\text{d}t)\exp\{C\int_{0}^{T}\|\nabla
u\|_{L^{\infty}}\text{d}t\}\nonumber\\
\leq & C(\int_{0}^{T}\|\nabla^{2}u\|_{L^{6}}\text{d}t+1).
\end{align}
Applying the standard elliptic regularity result
$\|\nabla^{2}u\|_{L^{6}}\leq C\|\Delta u\|_{L^{6}}$, H\"{o}lder
inequality, Sobolev embedding, the estimates \eqref{eq2.10} and
\eqref{eq2.46}, we have
\begin{align}\label{eq2.52}
\|\nabla^{2}u\|_{L^{6}}\leq& C(\|\rho u_{t}\|_{L^{6}}+\|\rho
u\cdot\nabla u\|_{L^{6}}+\|\nabla P\|_{L^{6}}+\|(\nabla
d)^{T}(\Delta d-f(d))\|_{L^{6}})\nonumber\\
\leq &C(\|\nabla u_{t}\|_{L^{2}}+\|u\|_{L^{\infty}}\|\nabla
u\|_{L^{6}}+\|\nabla \rho\|_{L^{6}}+\|\nabla
d\|_{L^{\infty}}\|\Delta d\|_{L^{2}}+\|\nabla
d\|_{L^{6}}\|f(d)\|_{L^{\infty}})\nonumber\\
\leq& C(\|\nabla u_{t}\|_{L^{2}}+\|u\|_{H^{2}}^{2}+\|\nabla
\rho\|_{L^{6}}+\|d\|_{H^{3}}^{2}+\|d\|_{H^{2}})\nonumber\\
\leq& C(\|\nabla u_{t}\|_{L^{2}}+\|\nabla \rho\|_{L^{6}}+1).
\end{align}
Inserting the estimate \eqref{eq2.52} into \eqref{eq2.51} yields
\begin{align*}
\|\nabla \rho\|_{L^{6}}\leq& C\int_{0}^{T}(\|\nabla
u_{t}\|_{L^{2}}+\|\nabla \rho\|_{L^{6}}+1)\text{d}t\\
\leq&C\int_{0}^{T}(\|\nabla u_{t}\|_{L^{2}}^{2}+\|\nabla
\rho\|_{L^{6}}+1)\text{d}t\leq
C\int_{0}^{T}\|\nabla\rho\|_{L^{6}}\text{d}t+C,
\end{align*}
where we have used the estimate \eqref{eq2.38}, then applying the
Gronwall's inequality gives
\begin{align}\label{eq2.53}
\sup_{0\leq t\leq T}\|\nabla \rho\|_{L^{6}}\leq C.
\end{align}
From \eqref{eq2.52} and \eqref{eq2.53}, we have
\begin{align}\label{eq2.54}
\int_{0}^{T}\|\nabla^{2} u\|_{L^{6}}^{2}\text{d}t\leq
C(\int_{0}^{T}\|\nabla u_{t}\|_{L^{2}}^{2}\text{d}t+\sup_{0\leq
t\leq T}\|\nabla \rho\|_{L^{6}}^{2}+C)\leq C.
\end{align}
It is easy to known that the estimate \eqref{eq2.48} follows
\eqref{eq2.53} and \eqref{eq2.54} immediately. This completes the
proof of Lemma \ref{lem2.6}.
\end{proof}
\medskip

We now give the proof of Theorem \ref{thm1.2}.
\medskip

\textbf{Proof of Theorem \ref{thm1.2}.} From the existence result of
Theorem \ref{thm1.1}, we know that $\|u(t)\|_{H^{2}}$,
$\|\rho(t)\|_{W^{1,6}}$, $\|d(t)\|_{H^{3}}$ and
$\|\rho^{\frac{1}{2}}u_{t}(t)\|_{L^{2}}$ are all continuous on the
time interval $[0,T^{*})$. From the above Lemmas
\ref{lem2.1}--\ref{lem2.6}, we see that for all $T\in (0,T^{*})$,
\begin{align}\label{eq2.55}
&(\|u\|_{H^{2}},\|\rho\|_{W^{1,6}},\|d\|_{H^{3}},\|\rho^{\frac{1}{2}}u_{t}\|_{L^{2}})(T)\leq
C.
\end{align}
Furthermore, there hold
\begin{align}\label{eq2.56}
\rho^{\frac{1}{2}} u_{t}+\rho^{\frac{1}{2}} u\cdot\nabla u\in
L^{\infty}([0,T^{*}];L^{2}),
\end{align}
and for all $T\in (0,T^{*})$,
\begin{align}\label{eq2.57}
(\!\mu \Delta u\!-\!\lambda \operatorname{div} (\nabla d\odot\nabla
d\!-\!\frac{1}{2}(|\nabla d|^{2}+\!F(d)))-\!\nabla P)(T)=(\rho
u_{t}+\rho u\cdot\nabla u)(T)=\sqrt{\rho}g,
\end{align}
where $g(T)\triangleq (\rho^{\frac{1}{2}} u_{t}+\rho^{\frac{1}{2}}
u\cdot\nabla u)(\cdot,T)\in L^{2} $. Therefore, from \eqref{eq2.56}
and \eqref{eq2.57}, we can take $(\rho,u,d)|_{t=T}$ with any
$T\in(0,T^{*})$ as the initial data and apply Theorem \ref{thm1.1}
to extend the local strong solution to a time interval
$[T,T+\delta]$ for a uniform $\delta>0$ which only depends on the
bounds obtained in these lemmas, so that the solution can be
extended to the time interval $[0,T^{*}+\delta)$. This contradicts
with the maximality of $T^{*}$. Hence, the assumption
$\eqref{eq2.1}$ cannot be true. This completes the proof of Theorem
\ref{thm1.2}.\hfill$\Box$
\\
\\

\end{document}